\documentclass[11pt]{amsart}

\usepackage{amssymb,amsmath,amsthm,enumerate,latexsym, graphics,shapepar,nccmath,mathtools}
\usepackage{tikz-cd} 
\usepackage[all,2cell,dvips]{xy}
\usepackage[utf8]{inputenc}

\usepackage[colorlinks=true, pdfstartview=FitV, linkcolor=red, citecolor=blue, urlcolor=blue]{hyperref}

\usepackage{parskip}
\usepackage{cleveref}

\usepackage[margin=0.95in]{geometry}

\newtheorem{thm}{Theorem}[section]
\newtheorem{lemma}[thm]{Lemma}
\newtheorem{rmk}[thm]{Remark}

\newtheorem{cor}[thm]{Corollary}
\newtheorem{defn}[thm]{Definition}
\newtheorem{question}[thm]{Question}
\newtheorem{example}[thm]{Example}
\newtheorem{setup}[thm]{Setup}
\theoremstyle{definition}
\newtheorem{defrem}[thm]{Definition/Remark}

\newcommand{\im}{\mbox{\rm Im}}

\newcommand{\rank}{\mbox{\rm rank}}

\newcommand{\ass}{\mbox{\rm Ass}}

\newcommand{\depth}{\mbox{\rm depth}}

\newcommand{\pdim}{\mbox{\rm pdim}}
\newcommand{\reg}{\mbox{\rm reg}}

\newcommand{\G}{\mbox{$G_{\mathfrak{m}}$}}
\newcommand{\coker}{\mbox{\rm coker}}
\newcommand{\cmd}{\mbox{\rm{cmd}}}
\newcommand{\codim}{\mbox{\rm{codim}}}

\newcommand{\BF}{\mbox{$\mathbb F$}}

\newcommand{\BN}{\mbox{$\mathbb N$}}

\newcommand{\BZ}{\mbox{$\mathbb Z$}}

\newcommand{\m}{\mbox{$\mathfrak m$}}

\newcommand{\sk}{\mbox{$\mathsf{k}$}}

\begin{document}

\author[H. Ananthnarayan]{H. Ananthnarayan}
\address{Department of Mathematics, I.I.T. Bombay, Powai, Mumbai 400076.}
\email{ananth@math.iitb.ac.in}

\author[Manav Batavia]{Manav Batavia}
\address{Department of Mathematics, Purdue University, West Lafayette, IN 47907.}
\email{mbatavia@purdue.edu}

\author[Omkar Javadekar]{Omkar Javadekar}
\address{Department of Mathematics, I.I.T. Bombay, Powai, Mumbai 400076.}
\email{omkar@math.iitb.ac.in}

\begin{center}
\end{center}
\title{Syzygies of associated graded modules}



\begin{abstract}
Given a finitely generated module $M$ over a Noetherian local ring $R$, we give a characterization for the first syzygy of the associated graded module $\G(M)$ to be equigenerated. As an application of this, we identify a complex of free $\G(R)$-modules, arising from given free resolution of $M$ over $R$, which is a resolution of $\G(M)$ if and only if $\G(M)$ is a pure $\G(R)$-module. We also give several applications of the purity of $\G(M)$. Our results demonstrate that while not all algebraic properties of a module carry over to its associated graded module, the purity of the minimal free resolution of $\G(M)$ ensures that several important invariants are inherited.  In addition, we provide sufficient conditions for Cohen-Macaulayness and purity of $\G(M)$, and provide a local version of the Herzog-K\"uhl equations.
\end{abstract}
\maketitle

\section{Introduction}

Associated graded rings and modules provide a tool to study the properties of local rings and modules through their filtrations. It is commonly observed that if the associated graded ring $\G(R)$ satisfies a certain property (e.g. Cohen-Macaulay, Gorenstein, complete intersection), then so does $R$ (see \cite{Fr87}). But the transfer of these properties from $R$ to $\G(R)$ does not exhibit a nice behaviour. Therefore,  the study of properties such as Cohen-Macaulayness, Gorensteinness, Buchsbaumness when transitioning from a ring to its associated graded ring is a subject of considerable interest. Analyzing how the algebraic invariants such as the Betti numbers, depth, projective dimsension of modules transfer from the local to graded setup also helps in understanding the properties of the original ring and module. In the present article, we focus on understanding how the syzygies of modules over a Noetherian local ring relate to those of the associated graded modules. We identify conditions under which the invariants such as the projective dimension, Betti numbers, and the structure of maps in minimal free resolutions of modules over the local ring are preserved in passage to the associated graded module. 

Let $(R,\m, \sk)$ be a Noetherian local ring, $M$ be a finitely generated $R$-module, and $N$ denote the first syzygy module of $M$. A natural question arises: Is the first syzygy of $\G(M)$ isomorphic to a twist of $\G(N)$? Since $\G(N)$ is equigenerated, this also leads to the question of when the first syzygy module of $\G(M)$ is equigenerated. In this article, we show that both these questions are equivalent (Theorem \ref{thm:equiconditions}), and have positive answers if and only if every minimal generating set of $N$ is its standard basis (see Definition \ref{def:standard-basis}).
Motivated by this result, and the work of T.~Puthenpurakal (cf.~\cite{TP16}) over the power series ring, given a free resolution $\BF_{\bullet}$ of an $R$-module $M$, we construct a pure complex $\BF_{\bullet}^*$ of free modules over the associated graded ring $\G(R)$, and prove that $\G(M)$ has a pure $\G(R)$-resolution if and only if $\BF_{\bullet}^*$ is a resolution of $\G(M)$ (Theorem \ref{thm:inF-local}). In some special cases, the complex $\BF_{\bullet}^*$ coincides with the {\it linear part of} $\BF_{\bullet}$, appearing in the work of J.~Herzog and S.~Iyengar (cf.~\cite{HI07}).

In general, it is observed that many algebraic properties of an $R$-module $M$ do not carry forward to $\G(M)$. For example, even if $M$ is Cohen-Macaulay, $\G(M)$ need not be Cohen-Macaulay. The same applies to projective dimension, i.e., there even if $M$ has a finite projective dimension, the module $\G(M)$ need not have so. However, our results show that under the hypothesis of purity on the minimal free resolution of $\G(M)$, several invariants of $M$ are inherited by $\G(M)$. 

The Betti numbers of $M$ can be obtained from those of $\G(M)$ by \emph{negative consecutive cancellations}, following the work of M.~E.~Rossi and L.~Sharifan, when the ring is regular (cf.~\cite[Theorem 3.1]{RS10}), and A.~Sammartano in the general case (cf.~\cite[Theorem 2]{AS16}). In particular, it follows that if $\G(M)$ has a pure resolution, the Betti numbers of $\G(M)$ are equal to the Betti numbers of $M$. 
The latter statement also follows from Theorem \ref{thm:inF-local}. We would also like to remark that $\G(M)$ having a pure resolution is a rather strong condition. In particular, over Koszul rings, it poses severe restrictions on the minimal free resolution of $M$. The same can be observed for the fibre product rings (see Theorem \ref{thm:KoszulFibreProcuct}, Corollary \ref{cor:fibre}). We would like to mention that a more useful version of Theorem \ref{thm:inF-local} is its contrapositive, as it provides a criterion for determining whether $\G(M)$ has a pure resolution by checking whether $\BF^*_{\bullet}$ fails to be a minimal free resolution of $\G(M)$.

Next, we focus on applications of Theorem \ref{thm:inF-local}, especially to the modules of finite projective dimension. It is well-known that if $R$ has a finitely generated module $M$ of finite projective dimension, which is Cohen-Macaulay, then so is $R$. However, the proofs are quite involved, and use what are known as the \emph{intersection theorems} (see \cite[Section 9.4]{BH93} for more details). In this article, we prove the same result under the assumption that $\G(M)$ has a pure resolution, similar to \cite[Theorem 4.9]{AK18}, but using much simpler techniques.

When $R$ is a polynomial ring over a field, and $M$ is a graded $R$-module with pure resolution, J.~Herzog and H.~K\"uhl (\cite[Theorem 1] {HK84}) give a characterization for $M$ to be Cohen-Macaulay in terms of relations between Betti numbers of $M$, known as the Herzog-K\"uhl equations. M.~Boij and J.~S\"oderberg extended this result in \cite[Proposition 3.2]{BS12} to modules that are not necessarily pure. This was further generalized to give analogues of the Herzog-K\"uhl equations for modules over standard graded $\sk$-algebras by H.~Ananthnarayan and R.~Kumar (cf.~\cite[Theorems 3.6 and 3.9]{AK18}). In this article, we provide a local version of Herzog-K\"uhl equations (Theorem \ref{thm:HKlocal}).

This article is organized as follows. In Section 2, we introduce the notation, definitions, basic observations, and previous results that are needed in the rest of the article. 
Section 3 is devoted to the study of properties of $N^*$ (see Definition \ref{def:graded-objects}(b)), where $N$ is a submodule of a free $R$-module $F$. The key result of this section involves a characterization of $N^*$ being an equigenerated graded $\G(R)$-module (Theorem \ref{thm:equiconditions}).
Section 4 provides the culmination of the theory introduced in Sections 2 and 3. In Definition \ref{def:fstar}, we construct a complex $\BF_{\bullet}^{*}$ from a resolution $\BF_{\bullet}$ of $M$ and use Theorem \ref{thm:equiconditions} to prove that $\G(M)$ has a pure $\G(R)$-resolution if and only if $\BF_{\bullet}^*$ is a resolution of $\G(M)$ (Theorem \ref{thm:inF-local}). 
In Section 5, we focus on applications of our results to modules with finite projective dimension. We prove the local version of Herzog and K\"uhl's result (Theorem \ref{thm:cmdThm}). We give sufficient conditions for $\G(M)$ to be Cohen-Macaulay and have a pure resolution (Theorem \ref{thm:HKlocal}). 
We also provide sufficient conditions for $R$ to be Cohen-Macaulay (Theorem \ref{thm:R_is_cm}).

\section{Preliminaries}

\subsection{Graded Betti Numbers and Pure Resolutions}

\begin{defn} 
Let $\sk$ be a field and $R =  \bigoplus_{j \geq 0} R_j$ be standard graded $\sk$-algebra, i.e., $R_0 = \sk$ and $R = \sk[R_1]$, and $M =  \bigoplus_{j \in \mathbb{Z}} M_j$ be a graded $R$-module.
\begin{enumerate}[{\rm a)}]

\item The \emph{Hilbert series of $M$} is defined as $H_M(z)= \sum_{j\in \mathbb{Z}} \dim_{\mathsf k}(M_j) z^j$.

\item The \emph{$n$-twist} of a graded module $M$, denoted by $M(n)$, is the graded module defined as\\ $M(n)_j= M_{n+j}$ for all $j \in \BZ$.

\item Let $M, N$ be graded $R$-modules. Then an $R$-linear map $\phi:M \to N$ called a \emph{graded map of degree $n$} if $\phi(M_j) \subset N_{n+j}$ for all $j \in \BZ$. By convention, the term `graded map' means a graded map of degree zero.

\item {A finitely generated graded $R$-module $M$ is said to be {\emph{equigenerated}} if there exists $n \in \BZ$ such that $M = \langle x_1, \ldots, x_l\rangle$ with $\deg(x_i)=n$ for all $i$.}

\end{enumerate} 
\end{defn}

\begin{defn}{
 Let $M$ be a finitely generated graded $R$-module and 
\[\BF_{\bullet} : \cdots \xrightarrow[]{} F_n \xrightarrow[]{\phi_n} F_{n-1} \xrightarrow[]{\phi_{n-1}} \cdots \xrightarrow[]{\phi_1}F_0 \xrightarrow[]{\phi_0} M \xrightarrow[]{} 0\]
be a minimal free resolution of $M$ over $R$, i.e., for each $i$, $\phi_i$ is a graded map of degree zero, and $\ker(\phi_i) \subset R_+ F_i$. (Note that we often drop the module $M$ while writing its resolution $\BF_{\bullet}$).

\begin{enumerate}[{\rm a)}]
\item The module $\Omega_i^R(M)=\ker(\phi_{i-1})$ is a graded $R$-module, called the \emph{$i^{th}$ syzygy module} of $M$. The number of minimal generators of $\Omega_i^R(M)$ in degree $j$ is denoted by $\beta_{i,j}(M)$, and is called the  \emph{$(i,j)^{th}$ graded Betti number} of $M$. The number $\beta_i(M)=  \sum_{j}^{} \beta_{i,j}(M)$ is called the \emph{total $i^{th}$ Betti number of $M$}, and equals the minimal number of generators of $\Omega_i(M)$.

\item The series $\mathcal{P}^R_M(z)= \sum_{i\geq 0} \beta_i(M)z^i$ (or simply $\mathcal{P}_M(z)$) is called the \emph{Poincar\'e series of $M$}. Similarly, the series $\mathcal{P}^R_M(s,t)=\sum_{i,j} \beta_{i,j}(M) s^it^j$ is called the \emph{graded Poincar\'e series} of $M$.

\item 
The \emph{regularity of $M$} is defined as
\[ \reg(M) = \sup \{ j-i \mid \beta_{i,j}(M) \neq 0\}. \]

\item The resolution $\BF_{\bullet}$ is said to be \emph{pure} if for every $i$, $\beta_{i,j}(M) \neq 0$ for at most one $j$. In such a case, $M$ is said to be a \emph{pure module} of type 
\begin{enumerate}[{\rm 1)}]
\item ${\mathbf{\delta}}=(\delta_0, \delta_1,\delta_2,\ldots)$ if $\pdim(M)= \infty$ and $\beta_{i,\delta_i}(M) \neq 0$ for all $i \geq 0$. 

\item $\delta= (\delta_0, \delta_1, \ldots, \delta_p, \infty, \infty,\ldots)$ if $\pdim(M)=p$ and $\beta_{i,\delta_i}(M) \neq 0$ for $0 \leq i \leq p$.

\end{enumerate}

\item A pure module $M$, generated in degree $\delta_0$, is said to have a \emph{linear resolution} if $\beta_{i,j} \neq 0$ implies that $j = \delta_0 + i$. The ring $R$ is said to be a \emph{Koszul algebra} if $\sk$ has a linear resolution over $R$.

\end{enumerate}
}
\end{defn}

\begin{rmk}\label{rmk:HilbSeries}{\rm 
Let $M$ be a finitely generated graded $R$-module, then 
\begin{enumerate}[\rm a)] 
\item The following is well-known (e.g., see \cite[Section 4.1]{BH93}). There exists $f(z) \in \mathbb{Z}[z,z^{-1}]$ such that $H_M(z)= f(z)/(1-z)^d$, where $d= \dim(M)$ and $f(1) \in \mathbb N$.
\item If $M$ has a linear resolution of type $\delta = (\delta_0, \delta_0 + 1,\delta_0 + 2,\ldots)$, then it follows (e.g., from \cite[Remark 2.2(b)]{AK18}) that $H_M(z) = z^{\delta_0} H_R(z) \mathcal{P}_M(-z)$.
\item From the work of L. Avramov and I. Peeva (see \cite{AP01}) it is known that every finitely generated module over a Koszul algebra has finite regularity.
\end{enumerate}  
}\end{rmk}

\subsection{Cohen-Macaulay Defect and Multiplicity}\hfill{}

We assume we are working in the graded case in this subsection, but note that Definition \ref{def:cmd}, and the observations (a) and (b) in the Remark \ref{rmk:cmd} are also valid for finitely generated modules over a Noetherian local ring. Results analogous to Remark \ref{rmk:cmd}(c) and (d) are also true in the local case, as proved in Theorems \ref{thm:R_is_cm} and \ref{thm:cmdThm}.

\begin{defn}\label{def:cmd}
Let $R$ be a standard graded $\sk$-algebra, and $M$ be a finitely generated graded $R$-module. The Cohen-Macaulay defect of $M$, denoted $\cmd(M)$, is defined as
 $\cmd(M)=\dim(M)-\depth (M)$.
\end{defn}

We record some observations and known results related to  Cohen-Macaulay defect in the following remark.

\begin{rmk}\label{rmk:cmd}
{\rm Let $R$ and $M$ be as in the above definition. \hfill{}
\begin{enumerate}[{\rm a)}]
    \item $M$ is Cohen-Macaulay if and only if $\cmd(M)=0$. 
    \item If $\pdim_R(M)<\infty$, then $\cmd(M)=\cmd(R)$ if and only if $\codim(M)=\pdim_R(M)$.
    \item  (\cite[Proposition 3.7]{AK18}) If $M$ is pure, then $\codim(M)\leq \pdim_R(M)$. 
    
    Moreover, if $\pdim_R(M)< \infty$, then $\cmd(R)\leq \cmd(M)$. 
    \item (\cite[Theorem 3.9]{AK18}) Let $M$ be a pure $R$-module of type $\delta=(\delta_0, \ldots, \delta_p)$, and $b_i=(-1)^{i-1}\prod_{j\neq i}\frac{\delta_j-\delta_0}{\delta_{j}-\delta_i}$ for $i=1,\ldots, p$. Then $\cmd(M)=\cmd(R)$ if and only if $\beta_i=b_i\beta_0$ for $i=1,\dots,p$.
     
\end{enumerate}
}\end{rmk}

\begin{defn} Let $M$ be a graded $R$-module of dimension $d$ and $H_M(z)= f(z)/(1-z)^d$. Then the number $f(1)$ is called as the multiplicity of $M$, and we denote it by $e(M)$.
\end{defn}

\begin{rmk}\label{rmk:multiplicity}{\rm (\cite[Corollary 4.1]{AK18}) Let $M$ be a pure module of type $\delta = (\delta_0,\ldots, \delta_p)$ over a standard graded $k$-algebra $R$. If $\cmd(R) = \cmd(M)$, then 
       \[e(M) = e(R) \dfrac{\beta_0^R(M)}{p!} \prod_{j\neq i}(\delta_j-\delta_0) .\]

}\end{rmk}

The following is an observation which we use in Theorem \ref{thm:HKlocal}.

\begin{rmk}\label{rmk:additivityofmultiplicity}
{\rm Let $0 \to K \to M \to N \to 0$ be an exact sequence of graded $R$-modules. Then by the additivity of the Hilbert series over graded short exact sequences, we see that if $\dim(M)=\dim(N)$ and $e(M)=e(N)$, then $\dim(K)<\dim(M)$.

In particular, if $M$ is Cohen-Macaulay, then $K=0$ and $M\simeq N$. Indeed, this can be seen as follows:

Since $M$ is Cohen-Macaulay, $\dim(M)= \dim(R/\mathfrak{p})$ for every $\mathfrak{p}\in \ass(M)$.
Since $\ass(K) \subset \ass(M)$, and $\dim(K) = \max \{ \dim(R/\mathfrak{p}) \mid \mathfrak{p} \in \ass(K)\}$, we see that  $\ass(K)=\emptyset$, i.e., $K=0$.

}\end{rmk}

\subsection{Associated Graded Rings and Modules}
\begin{defn}\label{def:graded-objects}
 Given a Noetherian local ring $(R,\m,\mathsf k)$, and a finitely generated $R$-module $M$, let $\G(R)=\bigoplus_{i\geq 0}{\m^i}/{\m^{i+1}}$ and $\G(M)=\bigoplus_{i\geq 0}\m^{i}M/\m^{i+1}M$ be the corresponding associated graded objects. Note that $\G(R)$ is a standard graded $\sk$-algebra, and $\G(M)$ is a finitely generated graded $\G(R)$-module.
 
 \begin{enumerate}[{\rm a)}]
\item 
Given any nonzero element $x\in M$, define $\nu(x) = \min\{ i \mid x \in  \m^iM\setminus\m^{i+1}M\}$. \\ 
If $\nu(x)=i$, we define an element of degree $i$ in $\G(M)$ naturally associated to $x$  as follows:\\ define $x^* = x+ \m^{i+1}M \in \m^{i}M/\m^{i+1}M\subset \G(M)$. 

\item Given a nonzero submodule $N$ of $M$, we define order of $N$ as     
      $\nu(N) = \min \{ \nu(x) \mid x \in N \setminus \{ 0\}\}$.\\
Furthermore, we define a graded submodule of $\G(M)$, given by $N^*= \langle x^* \in \G(M) \mid x\in N\rangle$.
\item Let $\phi:R^m\to R^n$ be a non-zero $R$-linear map. Considering $\phi$ as a matrix in the free module $R^{mn}$, we define the \emph{initial form} of $\phi$, to be the corresponding matrix $\phi^*\in A^{mn}$. In other words, if $\phi=(a_{ij})$ and $\nu(\phi)=s$, then $\phi^*=(a_{ij}+\m^{s+1})$. 
\end{enumerate}
\end{defn} 

\begin{rmk}{\rm  Using the representation in the Definition \ref{def:graded-objects}(c), observe that
\begin{enumerate} [{\rm a)}]
    \item $\phi^*:\G(R^m) \to  \G(R^n)$ is a graded map of degree $s$.\\ 
    {\sc Notation:} We also use  $\phi^*$ to denote the induced map of degree zero from $\G(R^m)(-s-j)$ to $\G(R^n)(-j)$ for all $j\in \mathbb Z$.
    \item  If $\psi:R^k\to R^m$ is a non-zero $R$-linear map such that $\phi\circ\psi=0,$ then $\phi^*\circ \psi^*=0.$
    
\end{enumerate}
}\end{rmk}

\begin{question}\label{Que:standard-basis}
Let $N$ be a submodule of $M$. Suppose $\{y_1,\dots,y_k\}$ is a minimal generating set of $N$. Then $\langle y_1^*, \ldots, y_k^*\rangle \subset N^*$.  When does the equality hold?
\end{question}

 The following example shows that in general $N^* \not \subset \langle y_1^*,\dots,y_k^* \rangle$.

\begin{example}\label{example}{\rm
Let \[R= \mathsf{k}[[X,Y,Z]]/\langle XZ-Y^3, YZ-X^4, Z^2-X^3Y^2 \rangle.\]
Then \[\G(R)\simeq \mathsf{k}[x,y,z]/\langle xz, yz, z^2, y^4 \rangle.\]
Here, for $N= \langle X \rangle$, we have $N^* = \langle x, y^3\rangle$. So, $N^* \neq \langle X^*\rangle$.} 
\end{example}

\begin{defn}\label{def:standard-basis}
A subset $\{y_1,\dots,y_k\}$ of an $R$-module $N$ is said to be a \emph{standard basis} of $N$ if $\langle y_1^*,\dots,y_k^*\rangle=N^*.$
\end{defn}

\begin{defn}
    A Noetherian local ring $(R, \m, \mathsf k)$ is said to be \emph{Koszul} if $\G(R)$ is a Koszul $\mathsf k$-algebra.
\end{defn}
\subsection{Fibre Products}
\begin{defrem}\label{defrmk:fibreproductsCConnsum}\hfill{}
    \begin{enumerate}[{\rm a)}]
        \item Let $(R_1, \mathfrak m_1, \mathsf k)$ and $(R_2, \mathfrak m_2, \mathsf k)$ be Noetherian local rings (resp. standard graded $\mathsf k$-algebras), and $\pi_i:R_i \to \mathsf k$ denote the natural projections. Then we define the \emph{fibre product $R$ of $R_1$ and $R_2$ over $\mathsf k$} as
$$R = R_1 \times_{\mathsf k} R_2 = \{ (a,b)\in R_1 \times R_2 \mid \pi_1(a)=\pi_2(b)\}.$$
Note that in this case, the fibre product $R$ is also local (resp. a standard graded $\sk$-algebra) with the maximal ideal $\m\simeq \m_1 \oplus \m_2$.  On the other hand, if $(R, \m, \mathsf k)$ is a Noetherian local ring with $\m= \m_1 \oplus \m_2$ for proper ideals $\m_1$ and $\m_2$, then $R$ is the fibre product of $R/\m_1$ and $R/\m_2$ over $\mathsf k$. \\
In this article, we will only consider non-trivial fibre products $R_1 \times_{\mathsf k} R_2$, i.e., fibre products with $R_1 \neq \mathsf k \neq R_2$.
\item If $R=R_1 \times_{\mathsf k} R_2$, then $\G(R)= \G(R_1) \times_{\mathsf k} \G(R_2)$. Thus, if $R$ is a fibre product, then so is $\G(R)$. 
\item If $R_1$ and $R_2$ are Koszul, then so is $R_1 \times_{\mathsf k} R_2$ (see \cite{BF85} or \cite[Corollary 4.7]{AJK24}).
\item Let $R$ be a Koszul $\sk$-algebra which is a fibre product, and $M$ be a pure $R$-module. Then $\Omega_2^R(M)$ has a linear resolution (see \cite[Theorem 5.2]{AJK24}).
\item If $M$ is an $R_1$-module, then for every $i\geq 0$, we have $\Omega_i^{R_1}(M)\mid \Omega_i^{R}(M)$.

    \end{enumerate}
\end{defrem}

\section{Standard Basis of a Submodule of a Free Module}

\noindent
{\bf Notation:} In the rest of the article, 
$(R,\m,\sk)$ denotes a Noetherian local ring, $M$ a finitely generated $R$-module, and $A= \G(R)$.

\begin{setup}\label{setup:submodule}
We use the following notation (and immediate observations) throughout this section. 
\begin{enumerate}[{\rm a)}]
\item $F$ is a free $R$-module with basis $\{w_1,\ldots,w_l\}$, $N \subset \m F$ is a submodule, $M = F/N$ and $\pi : F\to M$ is the natural quotient map. If $x_i = \pi(w_i)$, then $\{x_1, \ldots, x_l\}$ is a minimal generating set of $M$ since $N \subset \m F$. 
\item The $A$-module $\G(M)$ is minimally generated by $\{x_1^*, \dots, x_l^*\}$. In particular, $\G(M)$ is generated in degree zero.
Let $\epsilon:\G(F)\to\G(M)$ be the natural surjective graded $A$-linear map of degree $0$, induced by $\pi$, defined as $\epsilon(w_i^*)=x_i^*.$ 
\item The $\m$-adic filtration on $F$ induces a filtration $\mathcal{F}=\{N_i=\m^iF\cap N\}_{i\in \mathbb{Z}}$ on $N$. The corresponding associated graded module is denoted $G_{\mathcal{F}}(N) = \bigoplus_{i\geq 0} ( N \cap \m^{i} F )/( N \cap \m^{i+1} F) $
\end{enumerate}
\end{setup}

Let the notation be as above. In the following two remarks, we show that $\ker(\epsilon) = N^* \simeq G_{\mathcal{F}}(N)$, where the last isomorphism is as graded $A$-modules. These facts are known, but we include their proofs for the sake of completeness.

\begin{rmk}\label{lemma:keris*}{\rm
With notation as in Setup \ref{setup:submodule}, we prove that $\ker(\epsilon) = N^*$ in this remark.
\begin{enumerate}[{\rm a)}]
\item Since $\epsilon$ is a graded map of degree $0$, we see that $\ker(\epsilon)$ is a graded submodule of $\G(F)$, with $i$th graded component $\frac{(\mathfrak{m}^{i}F\cap N)+\mathfrak{m}^{i+1}F}{\mathfrak{m}^{i+1}F}$. 
\item Let $y\in N$, and $\nu(y)=i$ in $F$. Then, $y^* = y + \m^{i+1}F$. Hence, $y \in \mathfrak{m}^{i}F\cap N$ forces $\epsilon(y^*)=0$ by (a). Thus, $N^*\subset \ker(\epsilon)$.
\item Let $y^* = y+\m^{i+1}F$ be a nonzero homogeneous element in $\ker(\epsilon)$. Since the $i$th graded component of $\G(F/N)$ is $({\mathfrak{m}^iF+N})/({\mathfrak{m}^{i+1}F+N})$, we have $y\in \m^{i+1}F+N$.

Write $y=y'+z$, where $y'\in\m^{i+1}F$ and $z\in N$. Thus, $z\in \m^iF\setminus \m^{i+1}F$ and $y^*=z^*.$ Hence, since $\ker(\epsilon)$ is graded, we get  $\ker(\epsilon)\subset N^{*}.$
\end{enumerate}
}\end{rmk}

\begin{rmk}\label{lemma:kernelepsilon}{\rm
Let the notation be as in Setup \ref{setup:submodule}. In this remark, we show that $G_{\mathcal{F}}(N) \simeq N^*$ as graded $A$-modules.
\begin{enumerate}[{\rm a)}]
\item The $i$th graded component of $G_{\mathcal{F}}(N)$ is $(\m^{i}F \cap N)/(\m^{i+1}F \cap N)$, which, by Remark \ref{lemma:keris*}, is isomorphic to the $i$th graded component of $N^*$ as an $R$-module.
\item The above isomorphism induces an additive function $\eta: G_{\mathcal{F}}(N) \to \G(F)$ defined as
\[\eta\left(\sum_{i\geq 0} y_i+(\m^{i+1}F \cap N) \right) = \sum_{i\geq 0} y_i + \m^{i+1}F\] where $y_i+ (\m^{i+1}F \cap N) \in (G_{\mathcal{F}}(N))_{i}$. It is clear that $\eta$ is one-one, with $\eta(G_{\mathcal{F}}(N)) = N^*$.
\item Let $a^*\in A$ be of degree $j$, and $\gamma = y+(\m^{i+1}F \cap N)\in G_{\mathcal{F}}(N)$ be nonzero. Then it follows that $\eta(a^* \gamma)= a^* \eta (\gamma)$ since $a^* \gamma= ay + (\m^{i+j+1}F \cap N)$. This fact, together with the additivity of $\eta$, shows that $\eta$ is $A$-linear.
\end{enumerate}
}\end{rmk}

\begin{lemma}\label{lemma:1to3}
With notation as in Setup \ref{setup:submodule}, let $\nu(N)=s$. If $N^*$ is equigenerated, then\\ $N \cap \m^iF = \m^{i-s}N$ for all $i\geq s$.
\end{lemma}
\begin{proof}
 Since $N^* = \langle y^* \mid y \in N\rangle$, and $\nu(N)=s$, we see that $N^*$ has a minimal generator in degree $s$. Thus, by the given hypothesis, it follows that $G_{\mathcal F}(N) \simeq N^*$ is generated in degree $s$. 
 
Set $N_i = N \cap \m^iF$ for all $i$. Since $\nu(N) = s$, we get $N_s = N$. 
Since the $i$th graded component of $G_{\mathcal F}(N)$ is $N_i/N_{i+1}$, and $G_{\mathcal F}(N)$ is generated in degree $s$, we get $$\frac{N_{s+i}}{N_{s+i+1}}=\mathfrak n^i\frac{N_s}{N_{s+1}}, \text{ i.e., } N_{s+i}=\m^iN_s+N_{s+i+1}=\m^iN+N_{s+i+1},$$ for $i\geq1.$ By the Artin-Rees lemma,  there exists $i_0$ such that $N_{s+i+1}=\m N_{s+i}$ for all $i\geq i_0$. Hence, for $i\geq i_0$, we get $N_{s+i}=\m^iN+\m N_{s+i}$, and therefore, by Nakayama Lemma, we see that $N_{s+i}=\m^i N$ for $i\geq i_0$. 

We show by descending induction that $N_{s+i}=\m^i N$ for all $i\leq i_0$. As seen above, this is true for $i=i_0$. Fix $i< i_0$, and assume $N_{s+i+1}=\m^{i+1}N$. Then, $\m^{i+1}N\subset \m N_{s+i}\subset N_{s+i+1}=\m^{i+1}N.$ Hence, $N_{s+i+1}=\m N_{s+i}$ and as before, Nakayama Lemma gives $N_{s+i}=\m^i N$. 

Thus, we have proved $N_{s+i}=\m^i N$ for all $i \geq s$, which was what we wanted. 
\end{proof}

The converse of the above result is true. In fact, we prove a stronger statement.

\begin{lemma}\label{lemma:3to1}
With notation as in Setup \ref{setup:submodule}, let $\nu(N)=s$. Suppose that $N \cap \m^i F = \m^{i-s}N$ for some $i>s$, then 
\begin{enumerate}[{\rm a)}]
\item for any minimal generating set $\{y_1, \ldots, y_k\}$ of $N$, we have $\nu(y_j)<i$ for all $j$, and 
\item no minimal generator of $N^*$ has degree $i$. 
\end{enumerate}  
\end{lemma}
\begin{proof}
(a) is true since the given hypothesis implies that $N\cap \m^iF \subset \m N$.

In order to prove (b), we show that $(N^*)_{i} \subset \mathfrak{n} N^*$. To see this, let $\gamma \in (N^*)_{i}$ be nonzero.

Note that $(N^*)_i=((\m^i F \cap N) + \m^{i+1} F)/ \m^{i+1} F= (\m^{i-s}N + \m^{i+1} F) / \m^{i+1} F$. 
Hence we may assume that $\gamma = y+  \m^{i+1} F = y^*$ for some $y \in \m^{i-s}N$. Writing $y= \sum\limits_{j=1}^{k}  a_j z_j$, where $a_j \in \m^{i-s}$ and $z_j \in N$, we get, 
\[ y +  \m^{i+1} F = \sum\limits_{j=1}^{k} (a_j+\m^{i-s+1})(z_j+\m^{s+1}F).\] 
Hence, we get $y^* = \sum a_j^*z_j^*\in \mathfrak{n} N^*$, where the sum is over pairs $(a_j,z_j)$, with $\nu(a_j) = i - s$, and $\nu(z_j) = s$. This completes the proof.
\end{proof}

The special case of $i=s+1$ in the previous lemma is interesting, which we record in the following:

\begin{lemma}{\label{partofmgs}}
Let the notation be as in Setup \ref{setup:submodule}. If $N\cap \m^{s+1}F= \m N$, where $\nu(N) = s$, then 
\begin{enumerate}[{\rm a)}]
\item Every minimal generator of $N$ has order $s$,  
\item $N^*$ does not have a minimal generator in degree $s+1$, and 
\item If $\{y_1, \ldots, y_k\}$ is a minimal generating set of $N$, then $\{y_1^*, \ldots, y_k^*\}$ forms a part of a minimal generating set of $N^*$.
\end{enumerate}  
Finally, if $\mu(N^*)= k$, then every minimal generating set of $N$ is a standard basis for $N$.
\end{lemma}
\begin{proof}
(a) and (b) are immediate from the previous lemma, and the last statement follows from (c).

Let $\alpha_1, \ldots, \alpha_k \in A$ be such that $\sum_j \alpha_j y_j^* \in \mathfrak{n} N^*$. In order to prove (c), it suffices to show that $\alpha_j\in \mathfrak{n}$ for all $j$. 

Writing $\alpha_j = \sum_{i\geq 0} \alpha_{ij}$, where $\alpha_{ij} \in A_i$ for all $j$, and using the fact that $\mathfrak n = \oplus_{i > 0} A_i$, we see that $\sum_{j=1}^k \alpha_j y_j^* \in \mathfrak{n} N^*$ implies $\sum_{j=1}^k \alpha_{0j} y_j^* \in \mathfrak{n} N^*$. Thus, replacing $\alpha_{0j}$ by $\alpha_j$, we now want to prove that if $\sum_{j=1}^k \alpha_j y_j^* \in \mathfrak{n} N^*$ for $\alpha_j \in A_0$, then $\alpha_j = 0$ for all $j$. 

Suppose this is not true. Reorder the generating set to assume $\alpha_j \neq 0$ if and only if $j \in \{1,\ldots, l\}$. Then for $1 \leq j \leq l$, we can write $\alpha_j = a_j^*$, $a_j \in R \setminus \m$, i.e., $\nu(a_j) = 0$. 

Now, we claim that $\sum_{j=1}^l \alpha_j y_j^* \in \mathfrak{n} N^*$ implies that $\sum_{j=1}^l a_j y_j \in N \cap \m^{s+1}F = \m N$. Assuming the claim, since $\{y_1\ldots, y_l\}$ is a part of a minimal generating set of $N$, this forces $a_j \in \m$, a contradiction to the fact that $\nu(a_j) = 0$. 

In order to prove the claim, note that, $$\sum_{j=1}^l \alpha_j y_j^*=\sum_{j=1}^l (a_j+\m)(y_j+m^{s+1}F)=\sum_{j=1}^l(a_jy_j+\m^{s+1}F)=\left(\sum_{j=1}^la_jy_j\right)+\m^{s+1}F \in \mathfrak{n} N^*.$$ Since $\nu(N)=s$, the degree $s$ component of $\mathfrak{n} N^*$ is 0, which implies that $\sum_{j=1}^l a_jy_j \in \m^{s+1}F$.  Thus, we have $\sum_{j=1}^l a_jy_j \in \m^{s+1}F\cap N=\m N$, which finishes the proof of the lemma.
\end{proof}

Given a finitely generated $R$-module $M$, one of our goals is to understand when $\G(M)$ has a pure resolution. Writing $M \simeq F/N$, where $F$ is a free $R$-module of finite rank mapping minimally onto $M$, Remark \ref{lemma:keris*} tells us that $N^*$ must be equigenerated for $\G(M)$ to have a pure resolution. The next theorem characterizes when the first syzygy module of $\G(M)$ is equigenerated.
In particular, we get some positive answers to Question \ref{Que:standard-basis}.

\begin{thm}\label{thm:equiconditions}
With notation as in Setup \ref{setup:submodule}, let $\nu(N)=s$. 
Then the following statements are equivalent:\\ 
{\rm(i)} $N^*$ is equigenerated (in degree $s$).\\
{\rm(ii)} $N^*\simeq \G(N)(-s)$.\\
{\rm(iii)} $N \cap \m^iF = \m^{i-s}N$ for all $i\geq s$.\\
{\rm (iv)} $N\cap \m^{s+1}F=\m N$ and $\mu(N^*)=k.$\\
{\rm(v)} Every minimal generating set $\{y_1,\ldots, y_k\}$ is a standard basis of $N$, satisfying $\nu(y_i)=s$ for all $i$.  \\
{\rm(vi)} There is a  minimal generating set $\{y_1,\ldots, y_k\}$ of $N$, which is a standard basis of $N$, and satisfies $\nu(y_i)=s$ for all $i$.
\end{thm}

\begin{proof}\hfill{}

{(i)$\Rightarrow$(iii)}: This implication is the content of Lemma \ref{lemma:1to3}.

{(iii)$\Rightarrow $(ii):} 
Note that by (iii), the filtration $\mathcal F = \{N \cap \m^iF\}$ is the same as the $\m$-adic filtration on $N$  shifted by $s$. Hence $G_{\mathcal F}(N) = \G(N)(-s)$. Now, (ii) follows by Remark \ref{lemma:kernelepsilon}.

{(ii)$\Rightarrow$(i)}: This implication follows since $\G(N)$ is generated in degree $0$.

(ii) $\Rightarrow$ (iv): Clearly, $\mu(N^*)=k$. Since $N^*$ is equigenerated, Lemma \ref{lemma:1to3} shows that $N\cap \m^{s+1}F=\m N$. 

(iv) $\Rightarrow$ (v): This implication is the content of Lemma \ref{partofmgs}.

(v) $\Rightarrow$ (vi) is clearly true.

(vi) $\Rightarrow$ (i): Since $\{y_1^*, \ldots, y_k^*\}$ is a generating set of $N^*$, with $\deg(y_j^*)=s$ for all $1 \leq j \leq k$, $N^*$ is equigenerated.
\end{proof}

Example \ref{example} shows that even if $R$ is Cohen-Macaulay and $N$ has a minimal generating set of elements with the same order, $N^*$ need not be equigenerated.  This shows that in statement (vi) of the above theorem, the condition $\{y_1,\ldots, y_k\}$ is a standard basis of $N$, is necessary.

\section{Associated Graded Modules with Pure Resolutions}

In this section, we use Theorem \ref{thm:equiconditions} to study the conditions for the associated graded module $\G(M)$ to have a pure resolution. 

\begin{lemma}
\label{lemma:epsilon}
With notation as in Setup \ref{setup:submodule}, set $F_0 = F$, $\phi_0 = \pi$, and $M = F_0/N$. Let $F_1$ be a free $R$-module such that $\phi_1: F_1 \to F_0$ maps minimally onto $N$. Then 
\begin{enumerate}[{\rm a)}]
\item $\epsilon$ is surjective, and $\epsilon\circ\phi_1^*=0$.

\item $\Omega_1^A(\G(M)) \simeq \ker(\epsilon) = N^*$.

 \item If $\Omega_1^A(\G(M))$ is equigenerated in degree $s$, then
 \begin{enumerate}[{\rm i)}]
     \item Every column of $\phi_1$ has order $s$.
     \item  $\im(\phi_1^*)=\ker(\epsilon) \simeq \G(N)(-s)$.
     \item $\phi_1^*:\G(F_1)(-s)\to \G(F_0)$ maps minimally onto $\ker(\epsilon)$.
 \end{enumerate} 
 In particular, $\Omega_1^A(\G(M)) \simeq \G(\Omega_1^R(M))(-s)$ and $\G(F_1)(-s) \xrightarrow[]{\phi_1^*} \G(F_0) \xrightarrow[]{\epsilon} \G(M) \to 0$ is exact.

\end{enumerate}
\end{lemma}
\begin{proof}
As observed in Setup \ref{setup:submodule}, the map $\epsilon$ is surjective as $\{x_1^*,\dots,x_l^*\}$ is a generating set of $\G(M)$. Since it is a minimal generating set, we see that $\Omega_1^A(\G(M)) \simeq \ker(\epsilon) = N^*$, proving (b). 

Let $\phi_1=(a_{ij})$ and $\nu(\phi_1)=s$. Then $\phi_1^*=(a_{ij}+\m^{s+1})$. Since $\phi_0\circ\phi_1=0$, we have $\sum_{i=1}^{l}a_{ij}x_i=0$ for all $j$. 
Therefore, \[\sum_{i=1}^{l} (a_{ij}+\m^{s+1})x_i^*=\sum_{i=1}^{l} (a_{ij}+\m^{s+1})(x_i+\m M)=\sum_{i=1}^{l} (a_{ij}x_i+\m^{s+1}M)=0\]
for all $j$. This proves that $\epsilon \circ \phi_1^*=0$, completing the proof of (a).

c) Since $\Omega_1^A(\G(M)) \simeq N^*$ is equigenerated in degree $s$, the isomorphism $\Omega_1^A(\G(M)) \simeq \G(N)(-s)$ follows by (b) and (i) $\Rightarrow$ (ii) of Theorem \ref{thm:equiconditions}. 

We know that $N$ is generated minimally by the columns of $\phi_1$, say $y_1,\ldots, y_k$. By (i) $\Rightarrow$ (v) of Theorem \ref{thm:equiconditions}, we get that $\{y_1, \ldots, y_k\}$ is a standard basis for $N$, and all $y_i$'s have the same order $s$. Thus, the columns of $\phi_1^*$ are $\{y_1^*, \ldots, y_k^*\}$, and in particular, $\im(\phi_1^*)=\langle y_1^*, \ldots, y_k^*\rangle = N^*$. 

Finally, since $\mu(N)= \mu(\G(N)(-s))$, we get that $\phi_1^*$ maps minimally onto $N^*$, which completes the proof.
\end{proof}

The last part of the above lemma says that if $\Omega^A_1(\G(M))$ is known to be equigenerated in degree $s$, then the exactness of $F_1 \xrightarrow[]{\phi_1} F_0 \xrightarrow[]{\phi_0} M \to 0$ forces that of $\G(F_1)(-s) \xrightarrow[]{\phi_1^*} \G(F_0) \xrightarrow[]{\epsilon} \G(M) \to 0$.\\

This, together with the ideas used in the proof above, motivate the following definition, which helps us understand when $\G(M)$ has a pure resolution.

\begin{defn}\label{def:fstar}
Let  $$\BF_{\bullet}: \ \cdots\to F_p\xrightarrow{\phi_{p}}\dots\to F_1\xrightarrow{\phi_1} F_0\to 0$$ be a free resolution of an $R$-module $M$, where $\nu(\phi_i) = s_i$ for $i \geq 1$.
We define a natural associated graded complex as follows:
$$\BF^*_{\bullet}: \ \cdots\to \G(F_p)(-\delta_p)\xrightarrow{\phi_{p}^*}\dots\to \G(F_1)(-\delta_1)\xrightarrow{\phi_1^*} \G(F_0)\to 0,$$
where, $\delta_i = \sum_{j=1}^i s_j$. 
\end{defn}

\begin{rmk}\label{rmk:cokernel}{\rm Let $C=\coker(\phi_1^*)$. 
\begin{enumerate}[{\rm a)}]
    \item If $\BF^*_{\bullet}$ is acyclic, then $C$ has a pure resolution of type $(\delta_0=0, \delta_1, \delta_2, \ldots)$ with $\beta^A_{i,\delta_i}(C)=\beta_i^R(M)$.
    
    \item With $\epsilon$ as in Setup \ref{setup:submodule}, we have $\epsilon \circ \phi_1^* =0$ by Lemma \ref{lemma:epsilon}. Therefore, $C$ maps onto $\G(M)$, and we have a short exact sequence $0 \to K \to C \to \G(M) \to 0$. In particular, if $K = 0$, we get $C \simeq \G(M)$.
\end{enumerate}
}\end{rmk}

\begin{question}\label{que:coker}
Let the notation be as above.
\begin{enumerate}[{\rm a)}]
    \item Is $\BF^*_{\bullet}$ acyclic?
    \item Is $C \simeq \G(M)$?
    \item If $\mathbb F_{\bullet}^*$ is acyclic, then is $C \simeq \G(M)$?
\end{enumerate}
\end{question}

{
If $\G(M)$ has a linear resolution, then it is known that Question \ref{que:coker}(a), (b) (and hence (c)) have positive answers (see \cite[Proposition 1.5]{HI07}). While the answer to Question \ref{que:coker}(c) is not known in general, under some special circumstances, we prove that it has a positive answer (see Theorem \ref{thm:HKlocal}).
In the next theorem, we see that both Questions \ref{que:coker}(a) and (b) have simultaneous positive answers if and only if $\G(M)$ has a pure resolution.
}
\begin{thm}
\label{thm:inF-local}
Let $(R,\m,\mathsf k)$ be a Noetherian local ring, $A = \G(R)$, and $M$ be a finitely generated $R$-module. Let $M$ be a finitely generated $R$-module with a minimal free resolution $\BF_{\bullet}$. Then $\G(M)$ has a pure resolution over $A$ if and only if $\BF^*_{\bullet}$ is a minimal $A$-free resolution of $\G(M)$.
 \end{thm}
\begin{proof} If $\BF^*_{\bullet}$ is a minimal free resolution of $\G(M)$, then by construction, it is pure.

In order to prove the converse, denote $\epsilon$ as $\phi_0^*$, let $\delta_0=0$, $K_0= \G(M)$, and $N_i=\im(\phi_i)$,  \\$K_{i+1}= \ker(\phi_{i}^*)$, for all $i\geq 0$. 
By induction on $i$, for all $i \geq 1$ we prove the following\\
{\it{Claim:}} 
\vspace*{-5pt}\begin{enumerate}[{\rm (i)}]
    \item  $0 \to K_{i} \to \G(F_{i-1})(-\delta_{i-1}) \xrightarrow[]{\phi_{i-1}^*}K_{i-1} \to 0$ is exact,
    \item $K_{i} \simeq \G(N_{i})(-\delta_{i}) \simeq \Omega_{i}(\G(M))$,
    \item $\phi_{i}^*$ maps minimally onto $K_{i}$.
\end{enumerate}
{\it{ Proof of claim.}} Since $\Omega_1^A(\G(M))$ is equigenerated, the statements (i)-(iii) hold for $i=1$ by Lemma \ref{lemma:epsilon} (c). Inductively assume that the statements (i)-(iii) hold for some $i\geq 1$. 

Then $\phi_i^*$ maps minimally onto $K_i$. The sequence  $0 \to K_{i+1} \to \G(F_{i})(-\delta_{i}) \xrightarrow[]{\phi_{i}^*}K_{i} \to 0$ is exact since $K_{i+1}= \ker(\phi_i^*)$. The fact that $K_i \simeq \Omega_i(\G(M))$ implies  that $K_{i+1} \simeq \Omega_{i+1}(\G(M))$. Furthermore, since $\G(M)$ has a pure resolution, $K_{i+1}$ is equigenerated.

Now, consider the short exact sequence $0 \to N_{i+1} \to F_i \xrightarrow[]{\phi_i} N_i \to 0$. Since $K_i \simeq \G(N_i)(-\delta_i)$, from the exact sequence above, we have $K_{i+1}= \Omega_1(\G(N_i))(-\delta_i)$. Since $K_{i+1}$ is equigenerated, so is $\Omega_1(\G(N_i))\simeq K_{i+1}(\delta_i)$. 

Hence, by Lemma \ref{lemma:epsilon}, we get that $K_{i+1}(\delta_i) \simeq \G(N_{i+1})(-s_{i+1})$, i.e., $K_{i+1} \simeq \G(N_{i+1})(-\delta_{i+1})$. 
Moreover, $\phi_{i+1}^*: \G(F_{i+1})(-s_{i+1})\to \G(F_i)$ maps minimally onto $K_{i+1}(\delta_i)$, or equivalently,\\ $\phi_{i+1}^*: \G(F_{i+1})(-\delta_{i+1})\to \G(F_i)(-\delta_i)$ maps minimally onto $K_{i+1}$. Hence, the claim is proved. 

By the claim, we have $\im(\phi_i^*) =  \ker (\phi_{i-1}^*)$ for all $i \geq 1$, i.e., $\BF^*_{\bullet}$ is exact. Moreover, since $\phi_i^*$ maps minimally onto $K_i$ for all $i \geq 0$, we get that  $\BF^*_{\bullet}$ is a minimal free resolution of $\G(M)$.
\end{proof}

\begin{rmk}{ \rm We would also like to mention that the above theorem can also be proved by combining the results proved in \cite{RS10} and \cite{AS16} as follows.

In \cite[Theorem 2.3]{RS10}, the authors prove that given a minimal graded free resolution $\mathbb G_{\bullet}$ of $\G(M)$, we can build up a resolution $\mathbb F_{\bullet}$ (not necessarily minimal) of $M$ and a special filtration $\mathcal F$ such that $G_{\mathcal F}(\mathbb F_{\bullet}) = \mathbb G_{\bullet}$. If $\G(M)$ has a pure resolution, then this special filtration $\mathcal F$ coincides with the $\m$-adic filtration. In particular, in this case we get that $G_{\mathcal{F}}(\mathbb F_{\bullet}) = \mathbb F^*_{\bullet}$. 
Moreover, $M$ and $\G(M)$ have the same Betti numbers if and only if the constructed resolution $\mathbb F_{\bullet}$ of $M$ is minimal.

 In \cite{AS16} it is proved that the Betti numbers of $M$ can be obtained from the Betti numbers of $M$ by \emph{negative consecutive cancellations}. If $\G(M)$ has a pure resolution, then there are no negative consecutive cancellations. As a consequence, we get that the Betti numbers of $M$ and $\G(M)$ are the same. 
 
 So, if $\G(M)$ has a pure resolution, then $\mathbb F_{\bullet}$ is a minimal free resolution of $M$, and we get that $\mathbb F_{\bullet}^* = G_{\mathcal F}(\mathbb F_\bullet)$ is a minimal free resolution of $\G(M)$. This proves Theorem \ref{thm:inF-local}.
 
Our approach towards Theorem \ref{thm:inF-local} is different. In \cite[Theorem 2.3]{RS10}, the authors start with a resolution of $\G(M)$ and build a resolution of $M$, whereas we begin with a minimal free resolution of $M$ and attempt to construct a resolution of $\G(M)$. As a consequence of Theorem \ref{thm:inF-local}, in particular, we also get that the Betti numbers of $M$ and $\G(M)$ are the same whenever $\G(M)$ has a pure resolution.
}\end{rmk}

Theorem \ref{thm:inF-local} can  be used to check the non-purity of the minimal free resolution of $\G(M)$. We illustrate this with the help of the example below.
\begin{example}{\rm 
Let $R=\sk[[t^4, t^5, t^{11}]] \simeq \mathsf{k}[[X,Y,Z]]/\langle XZ-Y^3, YZ-X^4, Z^2-X^3Y^2 \rangle$ and let $M=R/\langle X \rangle$. Then $\mathbb F_{\bullet}: 0 \to R\xrightarrow[]{[X]} R\to 0$ is a minimal $R$-free resolution of $M$, and we have $\pdim(M)=1$.
Note that $A=\G(R)\simeq \mathsf{k}[x,y,z]/\langle xz, yz, z^2, y^4 \rangle$, has depth zero. Since $X^*=x$ is a zerodivisor on $A$, we see that 
the complex $\mathbb F_{\bullet}^*: 0 \to A(-1) \xrightarrow[]{[x]} A \to 0$ is not a minimal free resolution of $\G(M)$. Hence, by Theorem \ref{thm:inF-local}, we conclude that $\G(M)$ does not have a pure resolution.
}\end{example}

\begin{cor} 
With notation as in Theorem \ref{thm:inF-local}, consider a minimal free resolution $$\BF_{\bullet}: \ \cdots\to F_p\xrightarrow{\phi_{p}}\dots\to F_1\xrightarrow{\phi_1} F_0 \to 0$$
of $M$, and let $\Omega_i= \Omega_i^R(M)$, $s_i= \nu(\phi_i)$, $\delta_0=0$, and $\delta_i = \sum_{j\leq i} s_j$ for all $i \geq 1$.
 Then the following are equivalent 
 \begin{enumerate}[{\rm i)}]
     \item $\G(M)$ has a pure resolution.
     \item For $i\in \BN$, we have  $\Omega_i\cap  \m^j F_{i-1} = \m^{j-s_i}\Omega_i$ for all $ j> s_i$. 
     \item For $i\in \BN$, we have  $\Omega_i\cap  \m^j F_{i-1} = \m^{j-s_i}\Omega_i$ for all $ s_i < j \leq r + i - \delta_{i-1}$, where $\reg_A(\G(M)) \leq r$.
 \end{enumerate} 
 
  If this happens, then $\G(M)$ is pure of type $\delta=(0, \delta_1, \delta_2, \ldots)$.
\end{cor}
\begin{proof}
(i) $\Rightarrow$ (ii): 
Suppose that $\G(M)$ has a pure resolution. Then by (the proof of) Theorem \ref{thm:inF-local}, $\mathbb F^*_{\bullet}$ is a minimal free resolution of $\G(M)$, and $\Omega_{i}^{A}(\G(M))\simeq \G(\Omega_i)(-\delta_i)$.
Hence, by Lemma \ref{lemma:1to3},  for all $j > \delta_{i+1} - \delta_i = s_i$, we have $\Omega_i\cap  \m^j F_{i-1} = \m^{j-s_i}\Omega_i$.

(ii) $\Rightarrow$ (iii) is obvious. 

(iii) $\Rightarrow$ (i): We induce on $i$ to prove that $\Omega_i(\G(M))$ is generated in degree $\delta_i$. If $\Omega_1 \neq 0$, then note that by Lemma \ref{lemma:3to1},
the hypothesis implies that $\Omega_1^A(\G(M))$ has no minimal generator in degree $j$ for all $s_1 < j \leq r + 1$. By definition of regularity, every minimal generator of $\Omega_1(\G(M))$ has degree at most $r+1$. Hence,  every minimal generator of $\Omega_1^A(\G(M))$ has degree $s_1=\delta_1$, which proves the result for $i=1$. 

Inductively assume that the result is true for some $i \geq 1$. Then $\Omega_i^A(\G(M))$ is generated in degree $\delta_i$. If $\Omega_{i+1} \neq 0$, then by Lemma \ref{lemma:3to1}, the hypothesis implies that $\Omega_{i+1}^A(\G(M)) \simeq \G(\Omega_{i+1})(-\delta_{i+1})$ has no minimal generator in degree $j$ for all $\delta_{i+1}=s_{i+1}+ \delta_i < j \leq r + (i+1)$.
By definition of regularity, every minimal generator of $\Omega_{i+1}(\G(M))$ has degree at most $r +(i+1)$. Hence, every minimal generator of $\Omega_{i+1}^A(\G(M))$ has degree $\delta_{i+1}$. 
Therefore, by induction, it follows that $\G(M)$ has  pure resolution.
\end{proof}

In the next corollary, we recover two known results (see \cite[Equation (1.1.1)]{LS13} and the corollary following Theorem 4 in \cite{Fr87}) identifying a class of $R$-modules with rational Poincar\'e series.
\begin{cor}
Let $(R,\m,\sk)$ be a Noetherian local ring, $A = \G(R)$, and $M$ a finitely generated $R$-module. If $\G(M)$ has a linear resolution over $A$, then $\mathcal{P}^R_{M}(z)$ is rational. \\
In particular, if $\G(R)$ is Koszul, then $\mathcal{P}^R_{\mathsf{k}}(z)$ is rational.
\end{cor}
\begin{proof}
Since $\G(M)$ has a linear $A$-resolution, by Remark \ref{rmk:HilbSeries}(b), we have $$\mathcal{P}_{G_{\mathfrak m}(M)}^{A}(z)= H_{G_{\mathfrak m}(M)}(-z)/ H_{A}(-z),$$ which is rational. Now, by Theorem \ref{thm:inF-local}, we have $\mathcal{P}_M^R(z) =\mathcal{P}_{G_{\mathfrak m}(M)}^{A}(z)$, which proves the result. 
\end{proof} 

The following result shows that the property of $\G(M)$ having a pure resolution is rare and imposes strong restrictions on the resolution of $M$.
\begin{thm}\label{thm:KoszulFibreProcuct}
    Let $(R,\mathfrak m, \mathsf k)$ be a Noetherian local ring and $A= \G(R)$. Suppose that $M$ is a finitely generated $R$-module 
    such that $\G(M)$ has a pure resolution. If $\mathbb F_{\bullet}:\cdots\to  F_{2}\xrightarrow{\phi_2}F_1\xrightarrow{\phi_1} F_0\to0 $ is a minimal free resolution of $M$, then the following hold:
    \begin{enumerate}[{\rm a)}]
        \item 
   If $R$ is Koszul, then $\Omega_i^{A}(\G(M))$ has a linear resolution for some $i\geq 0$, and for each $j > i$, the order of every column of the matrix of $\phi_j$ is $1$.
    \item  If $A$ is a fibre product of standard graded Koszul $\sk$-algebras, then $\Omega_2^{A}(\G(M))$ has a linear resolution, and for each $j >2$, the order of every column of the matrix of $\phi_j$ is $1$.
    \end{enumerate}

\end{thm}
\begin{proof}
    The first statement follows from Theorem \ref{thm:inF-local} and Remark \ref{rmk:HilbSeries}(c), whereas the second one follows from Theorem \ref{thm:inF-local} and Definition/Remark \ref{defrmk:fibreproductsCConnsum}(d).
\end{proof}

The next corollary provides a class of examples where the statement (b) of the above theorem is applicable.
\begin{cor}\label{cor:fibre}
    Let $(R_1, \mathfrak m_1, \mathsf k)$ and $(R_2, \mathfrak m_2, \mathsf k)$ be Koszul Noetherian local rings, $R=R_1\times_{\mathsf k}R_2$ and $\mathfrak{m}$ denote the maximal ideal of $R$. Let $M$ be an $R$-module such that $\G(M)$ has a pure resolution. Then $\Omega_2^{A}(\G(M))$ has a linear resolution, and for each $j >2$, the order of every column of the matrix of $\phi_j$ is $1$.  
\end{cor}
\begin{proof}
     Note that since $R_1$ and $R_2$ are Koszul, $\G(R_1)$ and $\G(R_2)$ are Koszul $\mathsf k$-algebras. Hence, by Definition/Remark \ref{defrmk:fibreproductsCConnsum}(b),(c),  $\G(R)=\G(R_1)\times_{\mathsf k} \G(R_2)$ is Koszul, and Theorem \ref{thm:KoszulFibreProcuct} applies.
\end{proof}

 The following example illustrates the contents of the above corollary and demonstrates how to construct modules $M$ for which $\G(M)$ does not have a pure resolution.
 
\begin{example}{\rm 
    Let $R_1=\sk[[x_1,x_2,x_3]]$, $R_2=\sk[[y_1,y_2,y_3]]$, and $R=R_1 \times_{\mathsf k} R_2$. Then we have $\G(R)= G_{\mathfrak m_1}(R_1)\times_{\mathsf k} G_{\mathfrak m_2}(R_2) \simeq \sk[x_1,x_2,x_3]\times_{\mathsf k} \sk[y_1,y_2,y_3]$. 
    
    Let $M=R_1/\langle x_1^2, x_2^2, x_3^2\rangle$ be an $R_1$-module. Then $G_{\mathfrak m_1}(M)\simeq \sk[x_1,x_2,x_3]/\langle x_1^2, x_2^2, x_3^2\rangle$ is of type $(0,2,4,6)$ over $G_{\mathfrak m_1}(R_1)$. Note that $G_{\mathfrak{m_1}}(M)\simeq\G(M)$ as $\G(R)$-modules.
    
    By Definition/Remark \ref{defrmk:fibreproductsCConnsum}(e), for every $i\geq 0$, $\Omega_i^{G_{\mathfrak m_1}(R_1)}(\G(M))\mid \Omega_i^{G_{\mathfrak m}(R)}(\G(M))$. Since the module $\Omega_2^{G_{\mathfrak m_1}(R_1)}(\G(M))$ does not have a linear resolution, we see that $\Omega_2^{G_{\mathfrak m}(R)}(\G(M))$ does not have a linear resolution. Hence, by Corollary \ref{cor:fibre}, we see that $\G(M)$ does not have a pure resolution over $\G(R)$.
}\end{example}

\section{Applications to Modules with Finite Projective Dimension}

This section presents several applications of Theorem \ref{thm:inF-local}.
We begin with the following generalization of results in the graded case, known as the Herzog-K\"uhl equations, to the local case (see \cite[Theorem 1]{HK84} and \cite[Theorem 3.9]{AK18}). 

\begin{thm}\label{thm:cmdThm}
Let $(R,\m,\mathsf k)$ be a Noetherian local ring, $A = \G(R)$, and $M$ be a finitely generated $R$-module, with $\pdim_R(M)=p<\infty$. Suppose $\G(M)$ has a pure resolution.
Let 
$$\BF_{\bullet}:0\to F_p\xrightarrow{\phi_p} F_{p-1}\to\dots\to F_1\xrightarrow{\phi_1} F_0\to0$$ be a minimal resolution of $M$ with $\beta_i=\rank(F_i)$. Then the following are equivalent: 
\begin{enumerate}[{\rm i)}]
    \item $\cmd(M)=\cmd(R)$.
    \item $\beta_i=b_i\beta_0$ for $i=1,\dots,p$, where $b_i=(-1)^{i-1}\prod_{j\neq i}\frac{\delta_j}{\delta_{j}-\delta_i}$.
    \item $\cmd(\G(M)) = \cmd(A)$.
\end{enumerate}
Furthermore, if any of the above equivalent statements hold, then $e(M)=e(R)\dfrac{\beta_0}{p!}\prod_{i=1}^{p}\delta_i$.
\end{thm}
\begin{proof}
Since $\G(M)$ has a pure resolution, by Theorem \ref{thm:inF-local}, we have $\BF^*_{\bullet}$ is a minimal pure resolution of $\G(M)$ of type $(0, \delta_1, \ldots, \delta_p, \infty, \infty, \ldots)$ with $\beta_i^{A}(\G(M))=\beta_i$. Thus, the equivalence of 
(ii) and (iii) follows by Remark \ref{rmk:cmd}(d).

(i) $\Rightarrow$ (iii): Since $\dim(M)=\dim(\G(M))$, by the Auslander-Buchsbaum formula, we have \begin{equation*}
\begin{split}
\dim(\G(M))-\depth(\G(M))&=\dim(M)-(\depth(A)-p)\\
&=\dim(M)-\depth(A)+\depth(R)-\depth(M)\\
&=\dim(R)-\depth(A)\\
&=\dim(A)-\depth(A),
\end{split}
\end{equation*} 
where the third and the fourth equalities follow since $\cmd(M)=\cmd(R)$, and $\dim(A)=\dim(R)$ respectively. Hence, $\cmd(\G(M))=\cmd(A)$, proving (iii).

(iii) $\Rightarrow$ (i): Note that (iii) implies $\dim(M)= \dim(R)- \depth(A)+ \depth(\G(M))$. Also, since $\pdim_R(M)= \pdim_A(\G(M))$, we have $\depth(M)=\depth(R)-\depth(A)+\depth(\G(M))$. The statement (i) follows by subtracting the second equality from the first.

Finally, if any of the conditions (i)-(iii) hold, then by Remark \ref{rmk:multiplicity}, we get $e(M)=e(R)\dfrac{\beta_0}{p!}\prod_{i=1}^{p}\delta_i$.
\end{proof}

In particular, if $R$ is Cohen-Macaulay, and $\G(M)$ has a pure resolution, we get the following:

\begin{cor}
With notations as above, let $R$ be Cohen-Macaulay, and $\G(M)$ have a pure resolution over $\G(R)$. Then $M$ is Cohen-Macaulay if and only if the Betti numbers of $M$ satisfy conditions: $\beta_i=b_i\beta_0$ for $i=1,\dots,p$, where $b_i=(-1)^{i-1}\prod_{j\neq i}\frac{\delta_j}{\delta_{j}-\delta_i}$.
\end{cor}

We now provide equivalent conditions for when the associated graded module $\G(M)$ has a pure resolution and is Cohen-Macaulay. The following theorem generalizes \cite[Theorem 1.4]{TP16} which was proved under the assumptions  $R=\sk[[x_1,\ldots, x_n]]$ and $M$ is Cohen-Macaulay.

\begin{thm}
\label{thm:HKlocal}
Let $(R,\m,\mathsf k)$ be a Noetherian local ring, $A = \G(R)$, and $M$ be a finitely generated $R$-module, with $\pdim_R(M)=p<\infty$. Let 
$$\BF_{\bullet}:0\to F_p\xrightarrow{\phi_p} F_{p-1}\to\dots\to F_1\xrightarrow{\phi_1} F_0\to0$$ be a minimal resolution of $M$ with $\beta_i=\rank(F_i)$. Then the following are equivalent: 
\begin{enumerate}[{\rm i)}]
\item $\G(M)$ has a pure resolution and is Cohen-Macaulay.
\item $A$ is Cohen-Macaulay and the following hold: 
\begin{enumerate}[{\rm(a)}]
     \item  $\BF^*_{\bullet}$ is acyclic.
     \item $\beta_i=b_i\beta_0$ for $i=1,\dots,p$, where $b_i=(-1)^{i-1}\prod_{j\neq i}\frac{\delta_j}{\delta_{j}-\delta_i}$.
     \item The multiplicity of $M$, $$e(M)=e(R)\frac{\beta_0}{p!}\prod_{i=1}^{p}\delta_i.$$
\end{enumerate}
\item $\G(M)$ has a pure resolution, and $A$ and $M$ are Cohen-Macaulay.
\end{enumerate}
\end{thm}
\begin{proof}

(iii) $\Rightarrow$ (i): Since $A$ is Cohen-Macaulay, so is $R$. (This assertion will also be used in the proof of (ii) $\Rightarrow$ (iii)). Since $\G(M)$ has a pure resolution and $M$ is Cohen-Macaulay, by (i) $\Rightarrow$ (iii) of Theorem \ref{thm:cmdThm}, we get that $\G(M)$ is Cohen-Macaulay.

(i) $\Rightarrow$ (ii): Since $\G(M)$ is  pure, Theorem \ref{thm:inF-local} implies that $\BF^*_{\bullet}$ is a minimal pure resolution of $\G(M)$ of type $(0, \delta_1, \ldots, \delta_p, \infty, \infty, \ldots)$ with $\beta_i^{A}(\G(M))=\beta_i$. In particular, (a) holds. 

Moreover, by Remark \ref{rmk:cmd}(c), we get $\cmd(A) \leq \cmd(\G(M)) = 0$, and hence $A$ is Cohen-Macaulay. Since $\G(M)$ is Cohen-Macaulay, the statements (b) and (c) hold by Theorem \ref{thm:cmdThm}.

(ii) $\Rightarrow$ (iii): Let $\BF^*_{\bullet}$ is acyclic. Set $C = \coker(\phi_1^*)$. Then $C$ is pure of type and $(0, \delta_1, \ldots, \delta_p, \infty, \infty, \ldots)$, with $\beta_{i,\delta_i}^A(C) = \beta_i$. 
Thus, the Betti numbers of $C$ satisfy (b). Hence, the fact that $A$ is Cohen-Macaulay, forces $C$ to be so, by Remark \ref{rmk:cmd}(d). Moreover, by Remark \ref{rmk:multiplicity}, and (c), we get $e(C) = e(M) = e(\G(M))$. Finally, observe that $\dim(\G(M))=\dim(M)\geq\depth(M)=\depth(R)-p=\dim(R)-p=\dim(C)$.

Using the above, we first show that $\G(M) \simeq C$. In order to prove this, recall, by Remark \ref{rmk:cokernel}(b), we have a short exact sequence $0 \rightarrow K \rightarrow C \rightarrow \G(M) \rightarrow 0$. In particular, $\dim(C) \geq \dim(\G(M))$, forcing equality. Since $e(C) = e(\G(M))$, by Remark \ref{rmk:additivityofmultiplicity}, we get $K = 0$, i.e., $\G(M) \simeq C$. In particular, $\G(M)$ has a pure resolution, and is Cohen-Macaulay. This last assertion forces $M$ to be Cohen-Macaulay, e.g., by (i) $\Leftrightarrow$ (iii) of Theorem \ref{thm:cmdThm}. 
\end{proof}

It is known that the existence of a finitely generated Cohen-Macaulay module of finite projective dimension forces the ring to be Cohen-Macaulay. In the next result, we prove a special case of this fact, using results proved in \cite{AK18}, which involve much simpler techniques. In particular, this result gives us local analogues of \cite[Proposition 3.7]{AK18}, and its consequence, which is a part of \cite[Theorem 4.9]{AK18}.

\begin{thm}\label{thm:R_is_cm}
Let $(R,\m,\mathsf k)$ be a Noetherian local ring, $A = \G(R)$, and $M$ be a finitely generated $R$-module with $\pdim_R(M)<\infty$. If $\pdim_R(M)=\pdim_A(\G(M))$, then 
\begin{enumerate}[{\rm a)}]
    \item $\codim(M)\leq \pdim(M)$.
    \item If $M$ is Cohen-Macaulay, then $R$ is Cohen-Macaulay. 
\end{enumerate}
In particular, statements (a) and (b) above hold if $\G(M)$ has a pure resolution.
\end{thm}
\begin{proof}

Since $\pdim_R(M)=\pdim_{A}(\G(M))$, by Remark \ref{rmk:cmd}(c),  we get $\codim(\G(M))\leq \pdim(\G(M))$. This proves (a) as $\codim(\G(M)) = \codim(M)$.

Now, suppose that $M$ is Cohen-Macaulay. By the Auslander-Buchsbaum formula and (a), we have \\$\depth(R)=\depth(M)+\pdim_R(M)=\dim(M)+\pdim_R(M)\geq \dim(M)+\codim(M)=\dim(R)$. Thus, $R$ is Cohen-Macaulay.

Finally, if $\G(M)$ has a pure resolution, then by Theorem \ref{thm:inF-local}, we have $\pdim_R(M)=\pdim_{A}(\G(M))$. Hence, the ``in particular'' part follows.
\end{proof}


\begin{thebibliography}{10}

\bibitem{AJK24}
H.~Ananthnarayan, O.~Javadekar, and R.~Kumar.
\newblock Betti cones over fibre products.
\newblock {\em arXiv preprint arXiv:2404.07297}, 2024.

\bibitem{AK18}
H.~Ananthnarayan and R.~Kumar.
\newblock Modules with pure resolutions.
\newblock {\em Comm. Algebra}, 46(7):3155--3163, 2018.

\bibitem{AP01}
L.~L. Avramov and I.~Peeva.
\newblock Finite regularity and {K}oszul algebras.
\newblock {\em Amer. J. Math.}, 123(2):275--281, 2001.

\bibitem{BF85}
J.~Backelin and R.~Fr\"{o}berg.
\newblock Koszul algebras, {V}eronese subrings and rings with linear
  resolutions.
\newblock {\em Rev. Roumaine Math. Pures Appl.}, 30(2):85--97, 1985.

\bibitem{BS12}
M.~Boij and J.~S\"{o}derberg.
\newblock Betti numbers of graded modules and the multiplicity conjecture in
  the non-{C}ohen-{M}acaulay case.
\newblock {\em Algebra Number Theory}, 6(3):437--454, 2012.

\bibitem{BH93}
W.~Bruns and J.~Herzog.
\newblock {\em Cohen-{M}acaulay rings}, volume~39 of {\em Cambridge Studies in
  Advanced Mathematics}.
\newblock Cambridge University Press, Cambridge, 1993.

\bibitem{LS13}
L.~M. \c{S}ega.
\newblock On the linearity defect of the residue field.
\newblock {\em J. Algebra}, 384:276--290, 2013.

\bibitem{Fr87}
R.~Fr\"{o}berg.
\newblock Connections between a local ring and its associated graded ring.
\newblock {\em J. Algebra}, 111(2):300--305, 1987.

\bibitem{HI07}
J.~Herzog and S.~Iyengar.
\newblock Koszul modules.
\newblock {\em J. Pure Appl. Algebra}, 201(1-3):154--188, 2005.

\bibitem{HK84}
J.~Herzog and M.~K\"{u}hl.
\newblock On the {B}etti numbers of finite pure and linear resolutions.
\newblock {\em Comm. Algebra}, 12(13-14):1627--1646, 1984.

\bibitem{TP16}
T.~J. Puthenpurakal.
\newblock On associated graded modules having a pure resolution.
\newblock {\em Proc. Amer. Math. Soc.}, 144(10):4107--4114, 2016.

\bibitem{RS10}
M.~E. Rossi and L.~Sharifan.
\newblock Consecutive cancellations in {B}etti numbers of local rings.
\newblock {\em Proc. Amer. Math. Soc.}, 138(1):61--73, 2010.

\bibitem{AS16}
A.~Sammartano.
\newblock Consecutive cancellations in {T}or modules over local rings.
\newblock {\em J. Pure Appl. Algebra}, 220(12):3861--3865, 2016.

\end{thebibliography}
\end{document}